\documentclass[11pt,a4paper]{article}
\usepackage[T1]{fontenc}
\usepackage[utf8]{inputenc}
\usepackage{lmodern}
\usepackage[margin=28mm]{geometry}
\usepackage{amsmath,amssymb,amsthm}
\usepackage{microtype}
\usepackage{xurl}
\usepackage[hidelinks]{hyperref}
\hypersetup{
  pdftitle={Subexponential upper bound on the number of rich words},
  pdfauthor={Josef Rukavicka},
  pdfsubject={Combinatorics on words; enumeration of rich words},
  pdfkeywords={rich words, palindromes, subexponential growth, iterated logarithm}
}
\newcommand{\doi}[1]{\href{https://doi.org/#1}{\nolinkurl{#1}}}
\newcommand{\N}{\mathbb{N}}

\newcommand{\logstar}{\ln^{*}}
\newtheorem{theorem}{Theorem}
\newtheorem{proposition}{Proposition}
\newtheorem{lemma}{Lemma}
\newtheorem{corollary}{Corollary}
\theoremstyle{remark}
\newtheorem{remark}{Remark}
\numberwithin{equation}{section}
\allowdisplaybreaks[1]
\title{Subexponential upper bound on the number of rich words}
\author{Josef Rukavicka\thanks{Department of Mathematics,
Faculty of Nuclear Sciences and Physical Engineering, Czech Technical University in Prague
(josef.rukavicka@seznam.cz).}}
\date{October 19, 2026}

\begin{document}
\maketitle

\begin{abstract}
Let $R_q(n)$ denote the number of rich words of length $n$ over a fixed
alphabet of size $q\geq2$.  It is known that $R_q(n)^{1/n}\to1$.
We establish a nonrecursive subexponential upper bound on the number of rich words, with its asymptotic growth quantified in terms of an iterated logarithm. For every $\lambda\in(1/2,1)$ and $\gamma>1$, there exists $c>0$
such that, for every positive function $\phi$ satisfying
$\phi(n)\to\infty$ and $n/\phi(n)\to\infty$, the functions
\[
 f(n)=\left[c\logstar\!\left(\frac{n}{\phi(n)}\ln q\right)\right]^{1/\gamma},
 \qquad
 B(n)=q^{\frac{n}{\phi(n)}+\frac{n}{(2\lambda)^{f(n)-1}}}
\]
satisfy $R_q(n)\leq B(n)$ for all sufficiently large $n$ and
$B(n)^{1/n}\to1$.  Here $\logstar$ is the iterated natural logarithm,
with value zero on $(0,1]$.  The proof iterates a counting inequality
based on palindromic factorization and controls the resulting
multiplicative constants using the iterated logarithm.
\end{abstract}


\section{Introduction}
\label{sec:introduction}

A finite word $w$ is called \emph{rich} if it contains exactly $|w|+1$
distinct palindromic factors, including the empty word.  Equivalently,
it contains $|w|$ distinct nonempty palindromic factors.  This is the
maximum possible number of distinct palindromic factors of a word of
length $|w|$; see, for example, \cite{GlJuWiZa}.  Rich words admit several
useful structural characterizations
\cite{BuLuGlZa2,GlJuWiZa}.

Let $R_q(n)$ denote the number of rich words of length $n$ over a fixed
alphabet with $q$ letters.  Guo, Shallit, and Shur \cite{GuShSh15}
investigated the binary enumeration problem and conjectured that, for
some positive function $g(n)\to\infty$,
\[
 R_2(n)=O\!\left(\left(\frac{n}{g(n)}\right)^{\sqrt n}\right).
\]
Rukavicka \cite{RukavickaRichWords2017} proved that, for every fixed
finite alphabet,
\begin{equation}
 \label{eq:known-limit-intro}
 \lim_{n\to\infty}R_q(n)^{1/n}=1.
\end{equation}
The proof uses a factorization of a rich word into distinct palindromes,
obtained by successively removing the longest palindromic suffix.

The \emph{palindromic length} of a word is the minimum number of
palindromes whose concatenation is that word; see
\cite{FrPuZa}.  Bounds on palindromic factorizations of rich words were
improved in \cite{RUKAVICKA202295}, and bounds on their factor and
palindromic complexity were obtained in
\cite{RukavickaFacCmplx2021}. 
Other aspects of richness include extensions and the coexistence of
prescribed factors
\cite{10.1007/978-3-319-66396-8_7,10.1007/978-3-030-28796-2_23,RUKAVICKA2021richext,Vesti2014},
overlapping and square-free factors
\cite{PELANTOVA20132432,VESTI201748}, and generalized highly potential
words \cite{AGO2021184}.

The aim of this paper is to give a quantitative, nonrecursive form of
a subexponential upper bound.  More precisely, for $q\geq2$,
$1/2<\lambda<1$, and $\gamma>1$, we prove that there exists
$c=c(q,\lambda,\gamma)>0$ such that, for every positive function $\phi$
with $\phi(n)\to\infty$ and $n/\phi(n)\to\infty$,
\begin{equation}
 \label{eq:main-intro}
 R_q(n)\leq
 q^{\frac{n}{\phi(n)}+\frac{n}{(2\lambda)^{f(n)-1}}},
 \qquad
 f(n)=\left[c\logstar\!\left(\frac{n}{\phi(n)}\ln q\right)\right]^{1/\gamma},
\end{equation}
for all sufficiently large $n$.  The $n$th root of the right-hand side
tends to one.  The constant $c$ is independent of $\phi$, although the
threshold for $n$ may depend on $\phi$.

The argument has two stages.  First, we construct constants $K_j$ such
that
\[
 R_q(n)\leq K_jq^{n/((2\lambda)^{j-1})}
 \qquad(n,j\geq1).
\]
A suitable choice of $j$ gives the recursive bound in
Theorem~\ref{hhdbh7638ju}.  We then show that, for every $\gamma>1$,
$\logstar K_j=O(j^\gamma)$.  This allows us to choose an admissible
integer index in terms of $\logstar(n\ln q/\phi(n))$ and proves
\eqref{eq:main-intro} in Theorem~\ref{bbc8d783hjfd}.
The resulting bound has unspecified constants; it is not asserted to
be a sharp estimate or a fully evaluated numerical bound.

Section~\ref{sec:counting} records the required facts about rich words
and establishes a counting inequality valid at every length $n\geq2$.
Section~\ref{sec:iteration} gives the uniform improvement step and the
recursive bound.  Section~\ref{sec:logstar} controls the constants
$K_j$, and Section~\ref{sec:main} proves the nonrecursive bound.

\section{Preliminaries and a counting inequality}
\label{sec:counting}

Write $\N_0=\{0,1,2,\ldots\}$ and $\N_1=\{1,2,3,\ldots\}$.
Let $\Sigma$ be a fixed finite alphabet of cardinality $q$.  For a word
$w=w_1\cdots w_n$, write $|w|=n$ and $w^R=w_n\cdots w_1$.
A word is a palindrome if $w=w^R$.  A factor is a contiguous subword;
the empty word is both a factor and a palindrome.
All logarithms in this paper are natural logarithms.

For a unary alphabet there is exactly one word of each length, and it
is rich.  Thus $R_1(n)=1$ for all $n\in\N_0$.  We henceforth assume
$q\geq2$ and abbreviate $R_q(n)$ to $R(n)$.  In particular,
\begin{equation}
 \label{eq:trivial-bound}
 1\leq R(n)\leq q^n\qquad(n\in\N_1).
\end{equation}
Define
\[
 \Phi=\left\{\phi:[1,\infty)\to(0,\infty):
    \lim_{x\to\infty}\phi(x)=\infty,
    \quad\lim_{x\to\infty}\frac{x}{\phi(x)}=\infty\right\}.
\]
No continuity or monotonicity is assumed for functions in $\Phi$.
Only their values at positive integers will be used.

We recall the qualitative subexponentiality theorem.
\begin{theorem}[\cite{RukavickaRichWords2017}]
 \label{bn9d89825644}
 For every fixed finite alphabet,
 $\displaystyle\lim_{n\to\infty}R_q(n)^{1/n}=1$.
\end{theorem}

\begin{corollary}
 \label{d788354dn}
 For every real $a\geq1$, there is a constant $H(a)\geq1$ such that
 \[
  R(n)\leq H(a)q^{n/a}\qquad(n\in\N_1).
 \]
\end{corollary}
\begin{proof}
 Since $q>1$, Theorem~\ref{bn9d89825644} gives
 \[
  \lim_{n\to\infty}
  \left(\frac{R(n)}{q^{n/a}}\right)^{1/n}=q^{-1/a}<1.
 \]
 Consequently $R(n)/q^{n/a}\to0$.  This sequence is bounded, and we
 may take $H(a)$ to be the maximum of one and any upper bound for it.
\end{proof}

Every factor of a rich word is rich, and the longest palindromic suffix
of a nonempty rich word occurs exactly once in that word
\cite{GlJuWiZa}.  Successively deleting this suffix gives the
\emph{unioccurrent palindromic suffix factorization}, or
\emph{UPS-factorization},
\[
 w=w_p w_{p-1}\cdots w_1.
\]
Here each $w_i$ is the longest palindromic suffix of
$w_p\cdots w_i$.  These nonempty palindromes are pairwise distinct:
if $i<j$ and $w_i=w_j$, then $w_i$ would occur at least twice in
$w_p\cdots w_i$.
The factorization estimate in \cite{RukavickaRichWords2017} supplies a
constant $c_1=c_1(q)>0$ such that, for every rich word of length
$n\geq2$,
\begin{equation}
 \label{eq:ups-bound}
 p\leq c_1\frac{n}{\ln n}.
\end{equation}
Also $p\leq n$, since all factors are nonempty.

The next proposition is the counting step from
\cite{RukavickaRichWords2017}, with the cutoff explicitly restricted
to at most $n$.  We include the proof to specify its range of validity.
\begin{proposition}
 \label{iud9bn56s42}
 For $n\geq2$, let
 \[
  \kappa_n=\min\left\{n,
       \left\lceil c_1\frac{n}{\ln n}\right\rceil\right\}.
 \]
 Suppose $K\geq1$, $h>1$, and $R(r)\leq Kh^r$ for all
 $r\in\N_1$.  Then
 \begin{equation}
  \label{eq:counting-cutoff}
  R(n)\leq
  K^{\kappa_n}h^{(n+\kappa_n)/2}
  \left(\frac{en}{\kappa_n}\right)^{\kappa_n}
  \qquad(n\geq2).
 \end{equation}
\end{proposition}
\begin{proof}
 Fix $n\geq2$ and put $m=\kappa_n$.  For a fixed composition
 $n_1+\cdots+n_p=n$ with all $n_i\geq1$, each palindrome in a
 UPS-factorization with factor lengths $n_1,\ldots,n_p$ is determined
 by its prefix of length $\lceil n_i/2\rceil$.  These prefixes are rich.
 Therefore the number of words with these factor lengths is at most
 \[
  \prod_{i=1}^{p}R\!\left(\left\lceil\frac{n_i}{2}\right\rceil\right)
  \leq K^p h^{\sum_{i=1}^p\lceil n_i/2\rceil}
  \leq K^p h^{(n+p)/2}.
 \]
 There are $\binom{n-1}{p-1}$ such compositions, and
 \eqref{eq:ups-bound} gives $p\leq m$.  Since $K\geq1$ and $h>1$,
 \begin{equation}
  \label{eq:composition-count}
  R(n)\leq
  K^m h^{(n+m)/2}
  \sum_{p=1}^{m}\binom{n-1}{p-1}.
 \end{equation}
 We have $1\leq m\leq n$.  With $z=m/n\in(0,1]$, the binomial theorem
 implies
 \[
  z^m\sum_{i=0}^{m}\binom ni
  \leq\sum_{i=0}^{m}\binom ni z^i
  \leq(1+z)^n\leq e^{nz}=e^m.
 \]
 Hence
 \[
  \sum_{p=1}^{m}\binom{n-1}{p-1}
  \leq\sum_{i=0}^{m}\binom ni
  \leq\left(\frac{en}{m}\right)^m.
 \]
 Substitution into \eqref{eq:composition-count} proves the claim.
\end{proof}

\begin{corollary}
 \label{jjduf983nhdh}
 There is a constant $c_2=c_2(q)\geq1$ such that, on setting
 $\sigma_n=c_2 n/\ln n$ for $n\geq2$, the following holds:
 if $K\geq1$, $h>1$, and $R(r)\leq Kh^r$ for all $r\in\N_1$, then
 \begin{equation}
  \label{eq:smoothed-counting}
  R(n)\leq K^{\sigma_n}h^{(n+\sigma_n)/2}
             (c_2\ln n)^{\sigma_n}
  \qquad(n\geq2).
 \end{equation}
\end{corollary}
\begin{proof}
 Choose
 \[
  c_2\geq\max\left\{c_1+1,\frac{e}{c_1},
                          \frac{e}{\ln2},1\right\}.
 \]
 Since $n/\ln n\geq1$ for $n\geq2$, we have
 \[
  \kappa_n\leq c_1\frac{n}{\ln n}+1
             \leq(c_1+1)\frac{n}{\ln n}\leq\sigma_n.
 \]
 If $\kappa_n=n$, then $en/\kappa_n=e\leq c_2\ln n$.
 Otherwise $\kappa_n\geq c_1n/\ln n$, so
 $en/\kappa_n\leq e\ln n/c_1\leq c_2\ln n$.
 Finally $c_2\ln n\geq e>1$.  The result follows by increasing the
 bases and exponents in \eqref{eq:counting-cutoff}.
\end{proof}

\section{Iterative improvement and a recursive bound}
\label{sec:iteration}

Fix
\[
 \frac12<\lambda<1,\qquad \delta>1,\qquad s=2\lambda>1,
\]
and let $c_2$ and $\sigma_n$ be as in
Corollary~\ref{jjduf983nhdh}.  The next estimate is uniform in $K$.

\begin{lemma}
 \label{dxce88fbn29}
 There exists $c_3=c_3(q,\lambda,\delta)\geq1$ such that, whenever
 $\alpha>c_3$, $K\geq e$, and
 $n\in\N_1$ satisfies $n\geq\exp(\alpha^\delta\ln K)$, one has
 \begin{equation}
  \label{eq:uniform-improvement}
  K^{\sigma_n}q^{(n+\sigma_n)/(2\alpha)}
       (c_2\ln n)^{\sigma_n}
  \leq q^{n/(s\alpha)}.
 \end{equation}
\end{lemma}
\begin{proof}
 Set
 \[
  b=\ln q>0,\qquad
  D=\frac{(1-\lambda)b}{2\lambda c_2}>0.
 \]
 Taking logarithms in \eqref{eq:uniform-improvement} and dividing by
 $\sigma_n$ shows that it is equivalent to
 \begin{equation}
  \label{eq:bootstrap-logarithmic}
  D\frac{\ln n}{\alpha}-\ln\ln n
  \geq\ln K+\ln c_2+\frac{b}{2\alpha}.
 \end{equation}
 Since $\delta>1$, we can choose $c_3\geq1$ so that, for every
 $\alpha>c_3$,
 \begin{equation}
  \label{eq:choice-c3}
  D\alpha^{\delta-1}-2
  \geq\delta\ln\alpha+\ln c_2+\frac b2.
 \end{equation}
 This choice depends on $q,\lambda,\delta$, but not on $K$.

 Write $v=\ln K\geq1$, $t=\ln n$, and $t_0=\alpha^\delta v$.
 For $t\geq t_0$,
 \[
  \frac{d}{dt}\left(D\frac t\alpha-\ln t\right)
  =\frac D\alpha-\frac1t
  \geq\frac{D\alpha^{\delta-1}-1}{\alpha^\delta}\geq0,
 \]
 where the last inequality follows from \eqref{eq:choice-c3}.
 Thus it suffices to verify \eqref{eq:bootstrap-logarithmic} at
 $t=t_0$.
 At that point, the difference between the left-hand and right-hand
 sides is
 \begin{align*}
  &(D\alpha^{\delta-1}-1)v
       -\delta\ln\alpha-\ln v-\ln c_2-\frac b{2\alpha} \\
  &\quad\geq
    (D\alpha^{\delta-1}-2)v
       -\delta\ln\alpha-\ln c_2-\frac b{2\alpha} \\
  &\quad\geq
    (D\alpha^{\delta-1}-2)
       -\delta\ln\alpha-\ln c_2-\frac b{2\alpha}
    \geq0.
 \end{align*}
 Here we used $\ln v\leq v$, $v\geq1$, and
 \eqref{eq:choice-c3}.  This proves the estimate simultaneously for
 all $K\geq e$.  The assumed lower bound on $n$ also ensures $n\geq2$,
 so that $\sigma_n$ is defined.
\end{proof}

For $j\in\N_1$, let
\begin{equation}
 \label{eq:alpha-definition}
 \alpha_j=s^{j-1}.
\end{equation}
Fix a constant
\begin{equation}
 \label{eq:initial-constant}
 K_1>\max\{H(sc_3),e\},
\end{equation}
where $H$ is supplied by Corollary~\ref{d788354dn}.  For $j\geq2$,
define recursively
\begin{equation}
 \label{eq:K-recursion}
 \beta_j=\left\lceil
       \exp\!\left(\alpha_{j-1}^{\delta}\ln K_{j-1}\right)
          \right\rceil,
 \qquad K_j=q^{\beta_j}.
\end{equation}
These constants depend only on the fixed parameters and the choice
of $K_1$.  Since $\alpha_{j-1}\geq1$,
\begin{equation}
 \label{eq:K-increases}
 \beta_j\geq K_{j-1},\qquad
 K_j\geq2^{K_{j-1}}>K_{j-1}+1.
\end{equation}
In particular, $K_j\geq K_1>e$ and $K_j\to\infty$.

\begin{proposition}
 \label{nhdj7738dv}
 For all $j,n\in\N_1$,
 \begin{equation}
  \label{eq:iterated-bound}
  R(n)\leq K_jq^{n/\alpha_j}.
 \end{equation}
\end{proposition}
\begin{proof}
 We use induction on $j$.  For $j=1$, the conclusion follows from
 \eqref{eq:trivial-bound}, $\alpha_1=1$, and $K_1\geq1$.
 Suppose $j\geq2$ and that the conclusion holds for $j-1$.

 If $\alpha_j\leq sc_3$, then
 \[
  R(n)\leq H(sc_3)q^{n/(sc_3)}
       \leq K_jq^{n/\alpha_j}
       \qquad(n\in\N_1),
 \]
 by \eqref{eq:initial-constant} and \eqref{eq:K-increases}.

 Now suppose $\alpha_j>sc_3$.  Then
 $\alpha_{j-1}=\alpha_j/s>c_3$.
 Applying Corollary~\ref{jjduf983nhdh} to the induction hypothesis,
 with $K=K_{j-1}$ and $h=q^{1/\alpha_{j-1}}>1$, gives
 \[
  R(n)\leq K_{j-1}^{\sigma_n}
         q^{(n+\sigma_n)/(2\alpha_{j-1})}
         (c_2\ln n)^{\sigma_n}
         \qquad(n\geq2).
 \]
 If $n\geq\beta_j$, then
 $n\geq\exp(\alpha_{j-1}^{\delta}\ln K_{j-1})$, and
 Lemma~\ref{dxce88fbn29} implies
 \[
  R(n)\leq q^{n/(s\alpha_{j-1})}
         =q^{n/\alpha_j}
         \leq K_jq^{n/\alpha_j}.
 \]
 If $1\leq n<\beta_j$, the trivial bound instead gives
 \[
  R(n)\leq q^n\leq q^{\beta_j}=K_j
         \leq K_jq^{n/\alpha_j}.
 \]
 This completes the induction.
\end{proof}

\begin{theorem}
 \label{hhdbh7638ju}
 Let $\phi\in\Phi$.  For $n\in\N_1$, define
 \begin{align*}
  \tau(n)&=\max\left\{j\in\N_1:
       K_j\leq\max\{K_1,q^{n/\phi(n)}\}\right\},\\
  G(n)&=K_{\tau(n)}q^{n/\alpha_{\tau(n)}}.
 \end{align*}
 Then $\tau(n)$ is well defined,
 \[
  R(n)\leq G(n)\qquad(n\in\N_1),
  \qquad
  \lim_{n\to\infty}G(n)^{1/n}=1.
 \]
\end{theorem}
\begin{proof}
 The set defining $\tau(n)$ contains $1$ and is finite, because
 $K_j\to\infty$.  The upper bound follows immediately from
 Proposition~\ref{nhdj7738dv}.

 Since $n/\phi(n)\to\infty$ and $q>1$, every fixed index $j$ is
 admissible for all sufficiently large $n$.  Therefore
 $\tau(n)\to\infty$, and hence $\alpha_{\tau(n)}\to\infty$.
 By the definition of $\tau(n)$,
 \[
  0\leq\frac{\ln G(n)}{n}
  \leq\max\left\{\frac{\ln K_1}{n},
                  \frac{\ln q}{\phi(n)}\right\}
      +\frac{\ln q}{\alpha_{\tau(n)}}\longrightarrow0.
 \]
 Exponentiating proves the asserted limit.
\end{proof}

\section{Iterated logarithms and the growth of the constants}
\label{sec:logstar}

For $x>0$, let $\ell(x)=\logstar x$, where
\begin{equation}
 \label{eq:logstar-definition}
 \ell(x)=
 \begin{cases}
  0,&0<x\leq1,\\
  1+\ell(\ln x),&x>1.
 \end{cases}
\end{equation}
This recursive definition terminates after finitely many steps.
Define integer-height exponential towers by
\begin{equation}
 \label{eq:tower-definition}
 T_0=1,\qquad T_{m+1}=e^{T_m}\quad(m\in\N_0).
\end{equation}
Thus $T_m=e\uparrow\uparrow m$, but no tower of noninteger height is
used.

\begin{lemma}
 \label{lem:logstar-properties}
 The function $\ell$ is nondecreasing on $(0,\infty)$ and satisfies
 \[
  \ell(e^x)=1+\ell(x)\quad(x>0),\qquad
  \ell(T_m)=m\quad(m\in\N_0).
 \]
 For $x>1$,
 \begin{equation}
  \label{eq:tower-interval}
  T_{\ell(x)-1}<x\leq T_{\ell(x)}.
 \end{equation}
 In particular, if $x>0$, $m\in\N_0$, and $\ell(x)\leq m$, then
 $x\leq T_m$.  Also
 \begin{equation}
  \label{eq:logstar-elementary-growth}
  \ell(x)\leq1+\ln x\qquad(x\geq1).
 \end{equation}
\end{lemma}
\begin{proof}
 The identity for $\ell(e^x)$ follows directly from the definition.
 Induction on $m$ gives
 \[
  \ell(x)=m\quad\Longleftrightarrow\quad
  T_{m-1}<x\leq T_m
  \qquad(m\geq1).
 \]
 Together with $\ell(x)=0$ on $(0,1]$, this proves monotonicity,
 the tower identities, and \eqref{eq:tower-interval}.
 The stated consequence follows by comparing integer tower heights.

 To prove \eqref{eq:logstar-elementary-growth}, the case
 $\ell(x)\leq1$ is immediate.  If $m=\ell(x)\geq2$, then
 \eqref{eq:tower-interval} gives
 $\ln x>T_{m-2}$.  Since $T_r\geq r+1$ for every $r\geq0$,
 we get $\ln x>m-1$, as required.
\end{proof}

\begin{lemma}
 \label{uud983bdj}
 If $x,y>1$, then
 \[
  x+y\leq T_{\ell(x)+\ell(y)},\qquad
  xy\leq T_{\ell(x)+\ell(y)},\qquad
  x^y\leq T_{\ell(x)+\ell(y)}.
 \]
 Consequently,
 \begin{equation}
  \label{eq:logstar-arithmetic}
  \ell(x+y),\ \ell(xy),\ \ell(x^y)
  \leq\ell(x)+\ell(y).
 \end{equation}
\end{lemma}
\begin{proof}
 Put $a=\ell(x)$ and $b=\ell(y)$, so $a,b\geq1$ and
 $x\leq T_a$, $y\leq T_b$.
 For every $r\geq1$, we have $T_r\geq e$ and
 \[
  2T_r\leq T_r^2\leq e^{T_r}=T_{r+1}.
 \]
 Here $u^2\leq e^u$ for $u\geq e$ follows from $2\ln u\leq u$.
 With $r=\max\{a,b\}$, it follows that
 \[
  T_a+T_b\leq T_{r+1}\leq T_{a+b},\qquad
  T_aT_b\leq T_{r+1}\leq T_{a+b}.
 \]
 These prove the assertions for the sum and product.

 If $a=1$, then $(T_a)^{T_b}=e^{T_b}=T_{b+1}$.
 If $a\geq2$, the product estimate just proved gives
 \[
  (T_a)^{T_b}
  =e^{T_{a-1}T_b}
  \leq e^{T_{a+b-1}}=T_{a+b}.
 \]
 Thus $x^y\leq T_{a+b}$ as well.
 Applying Lemma~\ref{lem:logstar-properties} proves
 \eqref{eq:logstar-arithmetic}.
\end{proof}

\begin{proposition}
 \label{dty773hbd99jf}
 For every $\gamma>1$, there exists
 $c_6=c_6(q,\lambda,\delta,K_1,\gamma)\geq1$ such that
 \begin{equation}
  \label{eq:K-logstar-bound}
  \ell(K_j)\leq c_6j^\gamma\qquad(j\in\N_1).
 \end{equation}
 More precisely, if $b_0=\max\{1,\delta\ln s\}$, then
 \begin{equation}
  \label{eq:K-logstar-precise}
  \ell(K_j)\leq\ell(K_1)
    +(j-1)\bigl(2\ell(q)+1+\ell(b_0j)\bigr)
    \qquad(j\in\N_1).
 \end{equation}
\end{proposition}
\begin{proof}
 For $j\geq2$, \eqref{eq:K-recursion} and
 $\alpha_{j-1}\leq\alpha_j$ imply
 \[
  K_j
  =q^{\lceil K_{j-1}^{\alpha_{j-1}^{\delta}}\rceil}
  \leq q^{1+K_{j-1}^{\alpha_j^\delta}}
  =q\,q^{K_{j-1}^{\alpha_j^\delta}}.
 \]
 All arguments to which we apply Lemma~\ref{uud983bdj} exceed one;
 in particular, $\alpha_j^\delta>1$ when $j\geq2$.
 Repeated use of \eqref{eq:logstar-arithmetic} yields
 \begin{equation}
  \label{eq:logstar-recurrence}
  \ell(K_j)\leq2\ell(q)+\ell(K_{j-1})+\ell(\alpha_j^\delta).
 \end{equation}
 Moreover,
 \begin{equation}
  \label{eq:alpha-logstar}
  \ell(\alpha_j^\delta)
  =\ell\!\left(e^{(j-1)\delta\ln s}\right)
  =1+\ell\!\left((j-1)\delta\ln s\right).
 \end{equation}
 Iterating \eqref{eq:logstar-recurrence}, substituting
 \eqref{eq:alpha-logstar}, and using monotonicity gives
 \begin{align*}
  \ell(K_j)
   &\leq\ell(K_1)+2(j-1)\ell(q)
        +\sum_{i=2}^j\bigl(1+\ell((i-1)\delta\ln s)\bigr)\\
   &\leq\ell(K_1)
        +(j-1)\bigl(2\ell(q)+1+\ell(b_0j)\bigr).
 \end{align*}
 This also holds for $j=1$, and proves
 \eqref{eq:K-logstar-precise}.

 For $j>1$, if $b_0>1$ then Lemma~\ref{uud983bdj} gives
 $\ell(b_0j)\leq\ell(b_0)+\ell(j)$; if $b_0=1$, the same
 inequality holds as an equality.
 Hence \eqref{eq:K-logstar-precise} is
 $O\bigl(j(1+\ell(j))\bigr)$.
 By \eqref{eq:logstar-elementary-growth}, this is $O(j\ln j)$
 as $j\to\infty$, and $j\ln j=o(j^\gamma)$ for every $\gamma>1$.
 Enlarging a constant to include the finitely many remaining indices
 proves \eqref{eq:K-logstar-bound}.
\end{proof}

\section{A nonrecursive subexponential bound}
\label{sec:main}

Fix $\gamma>1$ and a constant $c_6$ as in
Proposition~\ref{dty773hbd99jf}.  The next lemma turns the iterated
logarithmic estimate into an admissible range of integer indices.

\begin{lemma}
 \label{hgd79bv3j4}
 Let $\phi\in\Phi$, and for $n\in\N_1$ write
 \[
  t(n)=\frac{n}{\phi(n)},\qquad
  L(n)=\ell\!\left(t(n)\ln q\right).
 \]
 If $n,j\in\N_1$ and $c_6j^\gamma\leq L(n)$, then
 \begin{equation}
  \label{eq:admissible-index}
  K_j<q^{t(n)}.
 \end{equation}
\end{lemma}
\begin{proof}
 Since $t(n)\ln q>0$, Lemma~\ref{lem:logstar-properties} gives
 \[
  \ell(q^{t(n)})=1+\ell(t(n)\ln q)=1+L(n).
 \]
 Applying \eqref{eq:tower-interval} to $q^{t(n)}>1$ yields
 $T_{L(n)}<q^{t(n)}$.
 On the other hand, the assumed inequality and
 Proposition~\ref{dty773hbd99jf} imply
 \[
  \ell(K_j)\leq c_6j^\gamma\leq L(n).
 \]
 The number $L(n)$ is an integer.  Therefore
 $K_j\leq T_{L(n)}<q^{t(n)}$, as claimed.
\end{proof}

\begin{theorem}
 \label{bbc8d783hjfd}
 Let $q\geq2$ be an integer, let $1/2<\lambda<1$, and let
 $\gamma>1$.  There exists a constant $c=c(q,\lambda,\gamma)>0$
 such that, for every $\phi\in\Phi$, there is
 $n_0=n_0(q,\lambda,\gamma,\phi)\in\N_1$ with the following property.
 Define, for $n\in\N_1$,
 \begin{equation}
  \label{eq:main-functions}
  \begin{split}
   f(n)&=\left[c\logstar\!\left(\frac{n}{\phi(n)}\ln q\right)
          \right]^{1/\gamma},\\
   B(n)&=q^{\frac{n}{\phi(n)}+\frac{n}{(2\lambda)^{f(n)-1}}}.
  \end{split}
 \end{equation}
 Then
 \[
  R_q(n)\leq B(n)\qquad(n\geq n_0),
  \qquad
  \lim_{n\to\infty}B(n)^{1/n}=1.
 \]
\end{theorem}
\begin{proof}
 In the construction of Sections~\ref{sec:iteration}
and~\ref{sec:logstar}, take $\delta=2$ and fix a choice of $K_1$
 satisfying \eqref{eq:initial-constant}.  These choices depend only on
 $q$ and $\lambda$.  Choose $c_6\geq1$ from
 Proposition~\ref{dty773hbd99jf}, and set
 \begin{equation}
  \label{eq:choice-main-c}
  c=\frac{1}{2^\gamma c_6}.
 \end{equation}
 In particular, $c$ is independent of $\phi$.

 Let $\phi\in\Phi$ and use $t(n)$ and $L(n)$ from
 Lemma~\ref{hgd79bv3j4}.  Since $t(n)\to\infty$ and $q>1$,
 we have $L(n)\to\infty$ and $f(n)=(cL(n))^{1/\gamma}\to\infty$.
 Choose $n_0$ so that $f(n)\geq1$ for every $n\geq n_0$.
 For such $n$, set
 \[
  j(n)=\lceil f(n)\rceil.
 \]
 Then $j(n)\geq1$ and $j(n)\leq f(n)+1\leq2f(n)$, whence
 \[
  c_6j(n)^\gamma
  \leq c_6\,2^\gamma f(n)^\gamma
  =L(n).
 \]
 Lemma~\ref{hgd79bv3j4} therefore gives
 $K_{j(n)}\leq q^{n/\phi(n)}$.
 Moreover, $j(n)\geq f(n)$ and $s=2\lambda>1$, so
 \[
  \alpha_{j(n)}=s^{j(n)-1}\geq s^{f(n)-1}.
 \]
 Proposition~\ref{nhdj7738dv} now implies
 \[
  R(n)\leq K_{j(n)}q^{n/\alpha_{j(n)}}
       \leq q^{\frac{n}{\phi(n)}+\frac{n}{s^{f(n)-1}}}
       =B(n).
 \]
 Finally,
 \[
  \frac{\ln B(n)}{n}
  =\ln q\left(\frac{1}{\phi(n)}+s^{1-f(n)}\right)
  \longrightarrow0,
 \]
 because $\phi(n)\to\infty$, $f(n)\to\infty$, and $s>1$.
 Thus $B(n)^{1/n}\to1$.
\end{proof}

A concrete choice of the auxiliary function removes $\phi$ from the
statement of the bound.
\begin{corollary}
 \label{cor:sqrt-choice}
 For every integer $q\geq2$, there exist $c_q>0$ and $n_q\in\N_1$
 such that, for all $n\geq n_q$,
 \[
  R_q(n)\leq B_q(n),\qquad
  B_q(n)=q^{\displaystyle
       \sqrt n+\frac{n}{(3/2)^{\sqrt{c_q\logstar(\sqrt n\ln q)}-1}}}.
 \]
 Moreover, $B_q(n)^{1/n}\to1$.
\end{corollary}
\begin{proof}
 Apply Theorem~\ref{bbc8d783hjfd} with
 $\lambda=3/4$, $\gamma=2$, and $\phi(x)=\sqrt x$.
\end{proof}

\begin{remark}
 The constant $c$ in Theorem~\ref{bbc8d783hjfd} is not evaluated
 numerically.  Its construction uses the finite starting constant
 $H(sc_3)$ from the known qualitative subexponentiality theorem.
 That input is used only to initialize the iteration; the subsequent
 control of $K_j$ yields the stated nonrecursive rate.
\end{remark}

\begin{remark}
 The bound is very slow to approach its asymptotic regime because it
 involves the iterated logarithm.  The present argument does not prove
 the conjectured binary estimate of Guo, Shallit, and Shur
 \cite{GuShSh15}.  In particular, it should be regarded as an
 asymptotic enumeration bound rather than as an estimate optimized for
 computations at moderate word lengths.
\end{remark}

\section*{Acknowledgments}

The author acknowledges the use of ChatGPT (OpenAI) in
reviewing mathematical arguments, identifying potential
gaps in proofs, suggesting corrections, and improving
the exposition of this manuscript.
The author takes full responsibility for the correctness
of the mathematical results and the final content of
the paper.


\end{document}